\documentclass[a4paper]{amsart}
\usepackage{graphicx}
\usepackage{amsmath}
\usepackage{amssymb}
\usepackage{oldgerm}
\usepackage{mathdots}
\usepackage{stmaryrd}
\usepackage{bm}
\usepackage[all]{xy}
\usepackage{color}
\usepackage{enumerate}
\usepackage{fourier}
\usepackage{hyperref}
\usepackage{mathrsfs}
\usepackage{tikz}
\usetikzlibrary{arrows,%
  decorations.markings,%
  matrix,%
  shapes}

\renewcommand{\Im}{\operatorname{Im}\nolimits}

\newcommand{\stmod}{\operatorname{\underline{mod}}\nolimits}
\newcommand{\Mod}{\operatorname{Mod}\nolimits}
\newcommand{\Ker}{\operatorname{Ker}\nolimits}

\newcommand{\cx}{\operatorname{cx}\nolimits}

\newcommand{\Fact}{\operatorname{\bf{Fact}}\nolimits}
\newcommand{\HFact}{\operatorname{\bf{HFact}}\nolimits}
\newcommand{\K}{\operatorname{\bf{K}}\nolimits}

\newcommand{\Kac}{\operatorname{\bf{K_{ac}}}\nolimits}

\newcommand{\C}{\mathscr{C}}
\newcommand{\F}{\mathscr{F}}

\newcommand{\s}{\Sigma}

\newtheorem{theorem}{Theorem}[section]

\newtheorem{corollary}[theorem]{Corollary}

\newtheorem{proposition}[theorem]{Proposition}

\theoremstyle{definition}

\theoremstyle{definition}

\newtheorem*{example}{Example}
\theoremstyle{definition}
\newtheorem{fact}[theorem]{Fact}

\newtheorem{remark}[theorem]{Remark}
\theoremstyle{remark}

\theoremstyle{definition}

\theoremstyle{definition}
\newtheorem*{notation}{Notation}
\theoremstyle{definition}

\begin{document}

\title{Matrix factorizations for quantum complete intersections}

\author{Petter Andreas Bergh and Karin Erdmann}

\address{Petter Andreas Bergh \\ Institutt for matematiske fag \\
NTNU \\ N-7491 Trondheim \\ Norway} \email{petter.bergh@ntnu.no}
\address{Karin Erdmann \\ Mathematical Institute \\ University of Oxford \\ Andrew Wiles Building \\ Radcliffe Observatory Quarter \\ Woodstock Road \\ Oxford \\ OX2 6GG \\ United Kingdom}
\email{erdmann@maths.ox.ac.uk}

\subjclass[2010]{16D50, 16E05, 16E35, 16G10, 16S80, 18E30}

\keywords{Twisted matrix factorizations, quantum complete intersections}

\thanks{The first author was partly supported by grant 221893 from the Research Council of Norway}

\begin{abstract}
We introduce twisted matrix factorizations for quantum complete intersections of codimension two. For such an algebra, we show that in a given dimension, almost all the indecomposable modules with bounded minimal projective resolutions correspond to such factorizations.
\end{abstract}

\maketitle

\section{Introduction}\label{sec:intro}

Matrix factorizations of elements in commutative rings were introduced by Eisenbud in \cite{Eisenbud}, in order to study free resolutions over the corresponding factor rings. In particular, he showed that minimal free resolutions over hypersurface rings eventually correspond to matrix factorizations, and are therefore eventually two-periodic. More precisely, let $Q$ be a regular local ring, $x$ a nonzero element, and denote the factor ring $Q/(x)$ by $R$. Eisenbud showed that if we take any finitely generated maximal Cohen-Macaulay module over $R$, without free summands, then its minimal free resolution is obtained from a matrix factorization of $x$ over $Q$.

The homotopy category of all matrix factorizations of $x$ over $Q$ forms a triangulated category in a natural way. The distinguished triangles are those that are isomorphic to the standard triangles constructed using mapping cones, as in the homotopy category of complexes over an additive category. Buchweitz remarked in \cite{Buchweitz} that in the above situation, with $Q$ regular, the homotopy category of matrix factorizations of $x$ is equivalent to the singularity category of $R$; this was proved explicitly by Orlov in \cite[Theorem 3.9]{Orlov}. 

\sloppy In this paper, we study twisted matrix factorizations for four dimensional quantum complete intersections of the form
$$A = k \langle x,y \rangle / (x^2, xy-qyx, y^2),$$
where $k$ is a field and $q$ is a nonzero element of $k$. Namely, for the algebra 
$$B = k \langle x,y \rangle / (x^2,y^2,xyx,yxy),$$ 
we construct a homotopy category $\HFact \left ( \F(B), S_{\nu}, \eta_w \right )$ of twisted matrix factorizations of the element $w=xy-qyx$, the twisting being with respect to the automorphism $\nu$ defined by $x \mapsto -q^{-1}x$ and $y \mapsto -qy$. By \cite{BerghJorgensen}, this homotopy category is triangulated in a natural way. It is related to the categories studied in \cite{CCKM}, whose twisted matrix factorizations are in some sense the complements of the ones we define. Namely, in \cite{CCKM} the factorizations are taken with respect to a regular element, whereas our element $w$ is actually a socle element in the algebra $B$. The resulting theories are quite different.

Similarly to the classical commutative case, reduction modulo the element $w$ induces a triangle functor
$$\HFact \left ( \F(B), S_{\nu}, \eta_w \right ) \longrightarrow \Kac \left ( \F (A) \right )$$
from this homotopy category of twisted matrix factorizations over $B$ into the homotopy category of acyclic complexes of free modules over the quantum complete intersection $A$. The latter is equivalent to the stable module category $\stmod A$, and so we obtain a triangle functor  
$$\HFact \left ( \F(B), S_{\nu}, \eta_w \right ) \longrightarrow \stmod A.$$
The image of a twisted matrix factorization is an $A$-module with a ``twisted two-periodic'' minimal free resolution. In particular, its minimal free resolution is bounded. In our main result, we show that when the field $k$ is algebraically closed, then almost all the indecomposable $A$-modules with bounded minimal projective resolutions are obtained this way, in the following sense. In a given dimension, all except two of the $A$-modules with such projective resolutions belong to the image of the triangle functor, and thus correspond to twisted matrix factorizations over $B$.

\section{Preliminaries}\label{sec:prelims}

In this section, we recall some of the theory from \cite{BerghJorgensen}. Fix an additive category $\C$, an additive automorphism $S
\colon \C \to \C$, and a central element $\eta$ in the suspended category $( \C, S )$. Thus $\eta$ is a natural transformation $1_{\C} \to S$ with
$$\eta_{S M} = S \eta_M$$ 
for all $M \in \C$. A \emph{$( \C, S )$-factorization} of $\eta$ is a sequence
$$\xymatrix@C=30pt{
M \ar[r]^{f} & N \ar[r]^{g} & SM}$$ 
of objects and morphisms in $\C$, satisfying 
$$g \circ f = \eta_M, \hspace{5mm} (S f) \circ g = \eta_N.$$
We denote such a factorization by $(M,N,f,g)$. A morphism 
$$\theta \colon (M_1,N_1,f_1,g_1) \to (M_2,N_2,f_2,g_2)$$
between factorizations is a pair $\theta = ( \psi, \phi )$ of morphisms in $\C$, with $\psi \colon M_1 \to M_2$ and $\phi \colon N_1 \to N_2$, such that the diagram
$$\xymatrix@C=30pt@R=20pt{
M_1 \ar[r]^{f_1} \ar[d]^{\psi} & N_1 \ar[r]^{g_1}  \ar[d]^{\phi} &
SM_1  \ar[d]^{S \psi} \\
M_2 \ar[r]^{f_2} & N_2 \ar[r]^{g_2} & SM_2}$$
commutes. Such a morphism is an isomorphism if both $\psi$ and $\phi$ are isomorphisms in $\C$. 

\sloppy The collection of all $( \C, S )$-factorizations of $\eta$, together with the morphisms just described, becomes an additive category $\Fact ( \C, S,
\eta )$. Its homotopy category $\HFact ( \C, S, \eta )$ has the same objects, but the morphisms are the homotopy equivalence classes $[ \theta ]$ of morphisms $\theta$ in $\Fact ( \C, S,\eta )$. Here, the notion of homotopy is the standard one. Namely, two morphisms
\begin{eqnarray*}
\theta = ( \psi, \phi ) \colon (M_1,N_1,f_1,g_1) \to (M_2,N_2,f_2,g_2) \\
\theta' = ( \psi', \phi' ) \colon (M_1,N_1,f_1,g_1) \to (M_2,N_2,f_2,g_2)
\end{eqnarray*}
in $\Fact ( \C, S, \eta )$ are homotopic if there exist diagonal morphisms
$$\xymatrix@C=50pt{
M_1 \ar[r]^{f_1} \ar[d]^{\psi'}_{\psi} & N_1 \ar[r]^{g_1}
\ar[d]^{\phi'}_{\phi} \ar[dl]_{s} & SM_1  \ar[d]^{S \psi'}_{S \psi} \ar[dl]_{t} \\
M_2 \ar[r]^{f_2} & N_2 \ar[r]^{g_2} & SM_2}$$
satisfying
\begin{eqnarray*}
\psi - \psi' & = & s \circ f_1 +(S^{-1}g_2) \circ (S^{-1}t) \\
\phi - \phi' & = & t \circ g_1 + f_2 \circ s.
\end{eqnarray*} 
Homotopies are compatible with all the operations in $\Fact ( \C, S, \eta )$, thus the homotopy category $\HFact ( \C, S, \eta )$ is also additive.

Given a $( \C, S )$-factorization $(M,N,f,g)$ of $\eta$, we obtain a new factorization $(N,SM,g,Sf)$ by rotating to the left. The suspension $\s (M,N,f,g)$ of $(M,N,f,g)$ is this new factorization with a sign change on the maps, i.e.\ $\s (M,N,f,g) = (N,SM,-g,-Sf)$:
$$\xymatrix@C=40pt{
(M,N,f,g) \colon & M \ar[r]^{f} & N \ar[r]^{g} & SM \\
\s (M,N,f,g) \colon & N \ar[r]^{-g} & SM \ar[r]^{-Sf} & SN}$$ 
This assignment induces an additive automorphism $\s$ on $\HFact ( \C, S, \eta )$. Next, consider a morphism $\theta$
$$\xymatrix@C=30pt@R=20pt{
M_1 \ar[r]^{f_1} \ar[d]^{\psi} & N_1 \ar[r]^{g_1}  \ar[d]^{\phi} &
SM_1  \ar[d]^{S \psi} \\
M_2 \ar[r]^{f_2} & N_2 \ar[r]^{g_2} & SM_2}$$
of $( \C, S )$-factorizations of $\eta$. Its mapping cone $C_{\theta}$ is the $( \C, S )$-factorization
$$\xymatrix@C=50pt{
N_1 \oplus M_2 \ar[r]^{\left ( \begin{smallmatrix} -g_1 & 0 \\ \phi & f_2 \end{smallmatrix} \right )} & S M_1 \oplus N_2 \ar[r]^{\left ( \begin{smallmatrix} -Sf_1 & 0 \\ S \psi & g_2 \end{smallmatrix} \right )} & S N_1 \oplus S M_2 }$$ 
of $\eta$. There are natural morphisms
$$i_{\theta} \colon (M_2,N_2,f_2,g_2) \to C_{\theta}, \hspace{5mm} \pi_{\theta} \colon C_{\theta} \to \s (M_1,N_1,f_1,g_1)$$
in $\Fact ( \C, S, \eta )$, and these are used to endow the suspended category $\left ( \HFact ( \C, S, \eta ), \s \right )$ with the structure of a triangulated category. Namely, given a morphism 
$$[ \theta ] \colon (M_1,N_1,f_1,g_1) \to (M_2,N_2,f_2,g_2)$$
in $\HFact ( \C, S, \eta)$, we declare the triangle
$$\xymatrix@C=30pt{
(M_1,N_1,f_1,g_1) \ar[r]^{[ \theta ]} & (M_2,N_2,f_2,g_2) \ar[r]^<<<<<<{[ i_{\theta} ]} & C_{\theta} \ar[r]^>>>>>>{[ \pi_{\theta} ]} & \s (M_1,N_1,f_1,g_1) }$$
in $\HFact ( \C, S, \eta)$ to be a standard triangle.

\begin{theorem}\cite{BerghJorgensen}\label{thm:triangulated}
Let $( \C, S)$ be a suspended additive category, $\eta$ a central element, and $\Delta$ the collection of all triangles in $\HFact ( \C, S, \eta)$ isomorphic to a standard triangle. Then $\left ( \HFact ( \C, S, \eta), \s, \Delta \right )$ is a triangulated category.
\end{theorem}

\section{Matrix factorizations}\label{sec:MF}

In this section, we establish some general theory of generalized matrix factorizations over arbitrary rings, not necessarily commutative. We shall work with the following setup.

\begin{notation}
Let $B$ be a ring, $w \in B$ a central element, and denote the factor ring $B/(w)$ by $A$. Furthermore, let $\nu \colon B \to B$ be an automorphism which fixes the elements of the ideal $(w)$, that is, $\nu (bw) = bw$ for all $b \in B$. 
\end{notation}

Note that the automorphism $\nu$ fixes the element $w$, and that it induces an automorphism on $A$: we denote also this by $\nu$. Now for every left $B$-module $M$, consider the twisted module ${_{\nu}M}$, on which $B$ acts via $\nu$, i.e.\ $b \cdot m = \nu (b) m$ for $b \in B$ and $m \in {_{\nu}M}$. The assignment $M \mapsto {_{\nu}M}$ induces an additive functor $S_{\nu} \colon \Mod B \to \Mod B$ on the category of left $B$-modules, a functor which acts as the identity on morphisms. This functor is an automorphism with inverse $S_{\nu^{-1}}$. Now for every $M \in \Mod B$, let $\eta_M \colon M \to S_{\nu}M$ be the multiplication map given by $m \mapsto wm$. Since $\nu (b)w = \nu (bw) = bw$ for all $b \in B$, the equalities
$$\eta_M(bm) = wbm = \nu (b) wm = b \cdot (wm) = b \cdot \eta_M(m)$$
hold, showing that the map $\eta_M$ is $B$-linear. Moreover, for every $B$-linear map $M \xrightarrow{f} N$ in $\Mod B$ the diagram
$$\xymatrix@C=30pt{
M \ar[r]^f \ar[d]^{\eta_M} & N \ar[d]^{\eta_N} \\
S_{\nu}M \ar[r]^{f=S_{\nu}f} & S_{\nu}N }$$
commutes, hence the collection 
$$\eta_w = \{ \eta_M \mid M \in \Mod B \}$$ 
forms a natural transformation $\eta_w \colon 1_{\Mod B} \to S_{\nu}$. It is easy to see that the coordinate maps commute with the functor $S_{\nu}$, that is, $\eta_{S_{\nu}M} = S_{\nu} \eta_M$ for all $M \in \Mod B$: this follows from the fact that $\nu (w) = w$. Hence the natural transformation $\eta_w$ is a central element in the suspended additive category $\left ( \Mod B, S_{\nu} \right )$.

We now follow the setup from Section \ref{sec:prelims}. Consider the homotopy category $\HFact \left ( \Mod B, S_{\nu}, \eta_w \right )$ of $\left ( \Mod B, S_{\nu} \right )$-factorizations of $\eta_w$, which is triangulated (in terms of standard triangles) by Theorem \ref{thm:triangulated}. By definition, its elements are sequences
$$\xymatrix@C=30pt{
M \ar[r]^{f} & N \ar[r]^{g} & {_{\nu}M}}$$ 
of left $B$-modules and $B$-linear maps, with
$$g \circ f = \eta_M, \hspace{5mm} f \circ g = \eta_N.$$
This is a noncommutative version of a category of matrix factorizations, and related to the categories studied in \cite{CCKM}, one important difference being that our element $w$ is not regular in $B$. The following result shows that reduction modulo $w$ induces a triangle functor from $\HFact \left ( \Mod B, S_{\nu}, \eta_w \right )$ to the homotopy category $\K( \Mod A)$ of complexes of left $A$-modules.

\begin{proposition}\label{prop:functor}
Let $B$ be a ring, $w \in B$ a central element, put $B/(w) = A$, and let $\nu \colon B \to B$ be an automorphism with $\nu (bw) = bw$ for all $b \in B$. Furthermore, let $S_{\nu} \colon \Mod B \to \Mod B$ be the twisting functor given by $\nu$, and $\eta_w \colon 1_{\Mod B} \to S_{\nu}$ the natural transformation with coordinates $\eta_M \colon M \to {_{\nu}M}$ given by $\eta_M(m) = wm$. Then reduction modulo $w$ induces a triangle functor 
$$\HFact \left ( \Mod B, S_{\nu}, \eta_w \right ) \longrightarrow \K \left ( \Mod A \right ).$$
The image of a factorization $(M,N,f,g)$ is the complex
$$\xymatrix@C=20pt{
\cdots \ar[r] & {_{\nu^{n-1}}(N/wN)} \ar[r]^{\overline{g}} & {_{\nu^n}(M/wM)} \ar[r]^{\overline{f}} & {_{\nu^n}(N/wN)} \ar[r]^{\overline{g}} & {_{\nu^{n+1}}(M/wM)} \ar[r] & \cdots}$$
of left $A$-modules and maps, with $M/wM$ in degree zero.
\end{proposition}

\begin{proof}
Let $(M,N,f,g)$ be a $\left ( \Mod B, S_{\nu} \right )$-factorization of $\eta_w$. Reduction modulo $w$ gives a sequence
$$\xymatrix@C=30pt{
M/wM \ar[r]^{\overline{f}} & N/wN \ar[r]^{\overline{g}} & {_{\nu}(M/wM)} }$$
of left $A$-modules and maps. Since $g \circ f = w \cdot 1_M$ and $f \circ g = w \cdot 1_N$, the sequence
$$\xymatrix@C=20pt{
\cdots \ar[r] & {_{\nu^{n-1}}(N/wN)} \ar[r]^{\overline{g}} & {_{\nu^n}(M/wM)} \ar[r]^{\overline{f}} & {_{\nu^n}(N/wN)} \ar[r]^{\overline{g}} & {_{\nu^{n+1}}(M/wM)} \ar[r] & \cdots}$$
is a complex. A homotopy between morphisms in $\Fact \left ( \Mod M, S_{\nu}, \eta_w \right )$ becomes a homotopy between complexes after reduction, hence we obtain an additive functor
$$T \colon \HFact \left ( \Mod B, S_{\nu}, \eta_w \right ) \longrightarrow \K( \Mod A),$$
to the homotopy category of complexes of left $A$-modules. The triangulated structures on the two categories $\HFact \left ( \Mod B, S_{\nu}, \eta_w \right )$ and $\K( \Mod A)$ are analogous: the distinguished triangles are those that are isomorphic to the standard triangles. Consequently, the functor $T$ is a triangle functor.
\end{proof}

In the next section, we shall mainly be dealing with free modules. The maps are then given by matrices, and the factorizations are noncommutative versions of classical matrix factorizations.

Let $\F (B)$ be the category of finitely generated free left $B$-modules. The twist ${_{\nu}F}$ of a free module $F$ is again a free module, hence the twisting functor $S_{\nu} \colon \Mod B \to \Mod B$ restricts to a functor $S_{\nu} \colon \F (B) \to \F (B)$. Similarly, the natural transformation $\eta_w \colon 1_{\Mod B} \to S_{\nu}$ restricts to a transformation $\eta_w \colon 1_{\F (B)} \to S_{\nu}$, which becomes a central element in the suspended additive category $\left ( \F (B), S_{\nu} \right )$. As in Proposition \ref{prop:functor}, reduction modulo $w$ induces a triangle functor $\HFact \left ( \F (B), S_{\nu}, \eta_w \right ) \to \K \left ( \F (A) \right )$. We record this in the following result.

\begin{proposition}\label{prop:functorfree}
Let $B$ be a ring, $w \in B$ a central element, put $B/(w) = A$, and let $\nu \colon B \to B$ be an automorphism with $\nu (bw) = bw$ for all $b \in B$. Furthermore, let $S_{\nu} \colon \F (B) \to \F (B)$ be the twisting functor given by $\nu$, and $\eta_w \colon 1_{\F (B)} \to S_{\nu}$ the natural transformation with coordinates $\eta_F \colon F \to {_{\nu}F}$ given by $\eta_F(m) = wm$. Then reduction modulo $w$ induces a triangle functor $\HFact \left ( \F (B), S_{\nu}, \eta_w \right ) \to \K( \F (A))$.
\end{proposition}

\section{Quantum complete intersections}\label{sec:QCI}

In this section, we apply the theory from the previous section to a class of quantum complete intersections. We fix the following notation throughout.

\begin{notation}
Let $k$ be a field and $B$ the noncommutative $k$-algebra  
$$B = k \langle x,y \rangle / (x^2, y^2, xyx,yxy).$$
Furthermore, let $q$ be a nonzero element in $k$, $w$ the quantum commutator $xy-qyx$ in $B$, and $A$ the algebra $B/(w)$. Thus $A$ is the four dimensional quantum complete intersection 
$$A = k \langle x,y \rangle / (x^2, xy-qyx, y^2).$$
\end{notation}

Our aim is to show that when $k$ is algebraically closed, then in a given dimension, almost all - in fact, all except two -  of the finitely generated indecomposable left $A$-modules with bounded minimal projective resolutions are obtained from twisted matrix factorizations of the element $w$ over $B$. The twisting automorphism $\nu \colon B \to B$ of interest acts on the generators by
$$\nu (x) = -q^{-1}x, \hspace{5mm} \nu (y) = -qy,$$
and is closely related to the Nakayama automorphism on $A$. Namely, by \cite[Lemma 3.1]{Bergh2}, the latter maps $x$ to $q^{-1}x$ and $y$ to $qy$. Note that $\nu (bw) = bw$ for all $b \in B$, hence we are in the setting from Section \ref{sec:MF}.

\begin{remark}\label{rem:automorphism}
The automorphism $\nu$ we have just defined is in some sense the only natural one when studying twisted matrix factorizations over $B$, at least among automorphisms $\mu \colon B \to B$ given by $\mu (x) = \alpha x$ and $\mu (y) = \alpha^{-1} y$ for some $\alpha \in k$. Namely, suppose that
$$\xymatrix@C=30pt{
B \ar[r]^{f} & B \ar[r]^{g} & {_{\mu}B}}$$ 
is a rank one factorization in $\HFact \left ( \F(B), S_{\mu}, \eta_w \right )$. Then the maps are given by right multiplication with elements in $B$, and these are of the form $c_0 + c_1x + c_2y + c_4xy + c_5yx$. Such an element is invertible in $B$ if and only if $c_0$ is nonzero, and then the factorization is zero in the homotopy category $\HFact \left ( \F(B), S_{\mu}, \eta_w \right )$. Moreover, the socle elements $xy$ and $yx$ play no role in the factorization. Consequently, we may assume that the maps $f$ and $g$ are given by right multiplication with elements $\beta_1x + \beta_2y$ and $\gamma_1x + \gamma_2y$, respectively. The compositions $g \circ f$ and $f \circ g$ both equal $\eta_B$, which is given by multiplication with $w$. It is not hard to see that this forces the scalars $\beta_1, \beta_2, \gamma_1$ and $\gamma_2$ to be nonzero, and that $\alpha$ must be $-q^{-1}$. Thus $\mu = \nu$.
\end{remark}

The canonical way of producing an $A$-module from a twisted matrix factorization over $B$ is to reduce modulo $w$ and then take the image of one of the maps. We shall now show that this assignment is functorial. By Proposition \ref{prop:functorfree}, reduction modulo $w$ induces a triangle functor 
$$\HFact \left ( \F (B), S_{\nu}, \eta_w \right ) \longrightarrow \K \left ( \F (A) \right ).$$
The following result shows that in our present setting, this functor maps the matrix factorizations, that is, the $\left ( \F (B), S_{\nu} \right )$-factorizations of $\eta_w$, to \emph{acyclic} complexes of free $A$-modules.

\begin{theorem}\label{thm:acyclic}
Let $k$ be a field and $B$ the $k$-algebra $k \langle x,y \rangle / (x^2, y^2, xyx,yxy).$ Take an element $0 \neq q \in k$, put $w=xy-qyx$, and consider the quantum complete intersection 
$$A = B/(w) \simeq k \langle x,y \rangle / (x^2, xy-qyx, y^2).$$
Furthermore, let $\nu$ be the automorphism on $B$ defined by $\nu (x) = -q^{-1} x$ and $\nu (y) = -qy$. Finally, let $\F (B)$ be the category of finitely generated free left $B$-modules, $S_{\nu} \colon \F(B) \to \F(B)$ the twisting functor given by $\nu$, and $\eta_w \colon 1_{\F(B)} \to S_{\nu}$ the natural transformation with coordinates $\eta_F \colon F \to {_{\nu}F}$ given by $\eta_F(m) = wm$. Then reduction modulo $w$ induces a triangle functor 
$$\HFact \left ( \F(B), S_{\nu}, \eta_w \right ) \longrightarrow \Kac \left ( \F (A) \right ),$$
where $\Kac( \F(A) )$ is the homotopy category of acyclic complexes of finitely generated free left $A$-modules. 
\end{theorem}

\begin{proof}
Consider a factorization
$$\xymatrix@C=30pt{
F \ar[r]^{f} & G \ar[r]^{g} & {_{\nu}F}}$$ 
in $\HFact \left ( \F(B), S_{\nu}, \eta_w \right )$. Since $g \circ f = \eta_F$ and $f \circ g = \eta_G$, the free $B$-modules $F$ and $G$ are of the same rank, say $r$. We may therefore assume that $F=G=B^r$, and that the maps are given by multiplying the standard generators of $B^r$ (that is, the row vectors with a single nonzero entry, the unit of $B$) on the right by $r \times r$ matrices $C= (c_{ij})$ and $D=(d_{ij})$. Thus for arbitrary elements the maps are given by
$$f \colon u \mapsto uC, \hspace{5mm} g \colon u' \mapsto \nu (u')D.$$
Here $\nu (u')$ is the row vector in $B^r$ obtained from $u'$ by applying the automorphism $\nu$ to all its entries. Note that we may assume that none of the matrices contain a unit: such an entry would imply that the factorization $(F,G,f,g)$ has a direct summand which is zero in the homotopy category $\HFact \left ( \F(B), S_{\nu}, \eta_w \right ) $.

Now take the standard generators $b_1, \dots, b_r$ of $B^r$. Applying $g \circ f$ to these gives
\begin{eqnarray*}
g \circ f (b_i) & = & g( b_iC ) \\
& = & g \left ( \sum_{j=1}^r c_{ij}b_j \right) \\
& = &  \sum_{j=1}^r \nu (c_{ij})b_jD \\
& = & \sum_{j=1}^r \nu (c_{ij}) \left ( \sum_{k=1}^r d_{jk}b_k \right ) \\
& = & \sum_{k=1}^r \left ( \sum_{j=1}^r \nu (c_{ij})d_{jk} \right ) b_k \\
& = & b_i \nu (C)D,
\end{eqnarray*}
where $\nu (C)$ is the matrix obtained from $C$ by applying the automorphism $\nu$ to all its entries. Since the composite map $g \circ f$ is given by multiplying the standard generators on the right by the matrix $\nu (C)D$, for an arbitrary element it is given by
$$g \circ f \colon u \mapsto \nu (u) \nu (C)D.$$
Similarly, the matrix for the composite map $S_{\nu}(f) \circ g$ is $DC$, so that for an arbitrary element it is given by
$$S_{\nu}(f) \circ g \colon u' \mapsto \nu (u')DC.$$ 
Now look at the matrices $C$ and $D$. Since none of their entries are units and the algebra $B$ has a basis $\{1,x,y,xy,yx \}$, we may decompose the matrices as
\begin{eqnarray*}
C & = & xC_1 + yC_2 + xyC_3 + yxC_4, \\
D & = & xD_1 + yD_2 + xyD_3 + yxD_4, 
\end{eqnarray*}
where the $C_i$ and $D_i$ are $r \times r$ matrices over $k$. This gives 
\begin{eqnarray*}
\nu(C)D & = & \left ( -q^{-1} xC_1 -q yC_2 + xyC_3 + yxC_4 \right ) \left ( xD_1 + yD_2 + xyD_3 + yxD_4 \right ) \\
& = & -q^{-1} xyC_1D_2 -q yxC_2D_1,
\end{eqnarray*}
which must equal $xyI - qyxI$, since $g \circ f = \eta_F$ and the matrix for this map is $(xy-qyx)I$, where $I$ is the $r \times r$ identlty matrix. Consequently, the matrices $C_1,C_2,D_1,D_2$ are invertible.

Let $u \in G$ be an element in $\Ker g$. We may write $u$ as $u = u_0 + xu_1 + yu_2 + xyu_3 + yxu_4$, where the $u_i$ are row vectors in $k^r$. Then
\begin{eqnarray*}
g(u) & = & \nu (u)D \\
& = & \left ( u_0 -q^{-1} xu_1 -q yu_2 + xyu_3 + yxu_4 \right ) \left ( xD_1 + yD_2 + xyD_3 + yxD_4 \right ) \\
& = & xu_0D_1 + yu_0D_2 + xyu_0D_3 + yxu_0D_4 -q^{-1} xyu_1D_2 -qyxu_2D_1,
\end{eqnarray*}
and this is zero in ${_{\nu}F}$. Since every element in ${_{\nu}F}$ can be written uniquely as a sum $u'_0 + xu'_1 + yu'_2 + xyu'_3 + yxu'_4$, where the $u'_i$ are row vectors in $k^r$, we see immediately that $u_0D_1 =0$. But $D_1$ is invertible, hence $u_0 =0$, and so
$$0 = -q^{-1} xyu_1D_2 -qyxu_2D_1$$
in ${_{\nu}F}$. This forces $u_1$ and $u_2$ to vanish as well ($u_1$ because the matrix $D_2$ is also invertible), giving $u = xyu_3 + yxu_4$. But then $u$ must belong to $\Im f$. Namely, since the matrices $C_1$ and $C_2$ are of rank $r$, there are row vectors $u'_1,u'_2 \in k^r$ with $u'_1C_1 = u_4$ and $u'_2C_2 = u_3$, giving
\begin{eqnarray*}
f( yu'_1 + xu'_2 ) & = &  \left ( yu'_1+xu'_2 \right ) \left ( xC_1 + yC_2 + xyC_3 + yxC_4 \right ) \\
& = & yxu'_1C_1 + xyu'_2C_2 \\
& = & yxu_4 + xyu_3 \\
& = & u.
\end{eqnarray*}
This shows that $\Ker g \subseteq \Im f$. 

Now take an element $u \in G$, and suppose that $u + wG$ belongs to $\Ker \overline{g}$ in the reduced sequence
$$\xymatrix@C=30pt{
F/wF \ar[r]^{\overline{f}} & G/wG \ar[r]^{\overline{g}} & {_{\nu}(F/wF)}}$$ 
Then $g(u) = wu = g \circ f (u')$ for some $u' \in F$, giving $g( u - f(u') ) = 0$. We showed above that $\Ker g \subseteq \Im f$, hence $u$ must belong to $\Im f$. It follows that the element $u + wG$ belongs to $\Im \overline{f}$, so the reduced sequence is exact. Similarly, the sequence
$$\xymatrix@C=25pt{
\cdots \ar[r] & {_{\nu^{n-1}}(G/wG)} \ar[r]^{\overline{g}} & {_{\nu^n}(F/wF)} \ar[r]^{\overline{f}} & {_{\nu^n}(G/wG)} \ar[r]^{\overline{g}} & {_{\nu^{n+1}}(F/wF)} \ar[r] & \cdots}$$
is acyclic, hence the image of the functor is contained in $\Kac( \F(A) )$.
\end{proof}

The following will be useful later.

\begin{remark}\label{rem:matrices}
As shown in the above proof, the $x$ and $y$ ``components'' of the two maps in a nonzero factorization in $\HFact \left ( \F(B), S_{\nu}, \eta_w \right )$ are invertible matrices. Namely, let
$$\xymatrix@C=30pt{
F \ar[r]^{f} & G \ar[r]^{g} & {_{\nu}F}}$$ 
be a nonzero rank $r$ factorization, with $f$ and $g$ defined in terms of $r \times r$ matrices $C$ and $D$, respectively. These matrices decompose as
\begin{eqnarray*}
C & = & xC_1 + yC_2 + xyC_3 + yxC_4, \\
D & = & xD_1 + yD_2 + xyD_3 + yxD_4, 
\end{eqnarray*}
where the $C_i$ and $D_i$ are $r \times r$ matrices over $k$. The four matrices $C_1,C_2,D_1,D_2$ were shown to be invertible.

Conversely, let $C_1$ and $C_2$ be invertible $r \times r$ matrices over $k$, and consider the $B$-linear maps $f \colon B^r \to B^r$ and $g \colon B^r \to {_{\nu}B^r}$ defined in terms of the matrices $xC_1+yC_2$ and $xC_2^{-1} -qyC_1^{-1}$, respectively. Then the matrix for the composition $g \circ f$ is
$$\nu \left ( xC_1+yC_2 \right ) \left ( xC_2^{-1} -qyC_1^{-1} \right ) = \left ( -q^{-1}xC_1 - qyC_2 \right ) \left ( xC_2^{-1} -qyC_1^{-1} \right ) = wI,$$
and for the composition $S_{\nu}(f) \circ g$ it is
$$\left ( xC_2^{-1} -qyC_1^{-1} \right ) \left ( xC_1+yC_2 \right ) = wI.$$
Thus $g \circ f = \eta_{B^r} = S_{\nu}(f) \circ g$, and so
$$\xymatrix@C=30pt{
B^r \ar[r]^{f} & B^r \ar[r]^{g} & {_{\nu}B^r}}$$ 
is an element of $\HFact \left ( \F(B), S_{\nu}, \eta_w \right )$.
\end{remark}

In the commutative case, when the element one factors out is a non-zerodivisor in the ring (and there is no twisting involved), the triangle functor from the homotopy category of matrix factorizations to the homotopy category of acyclic complexes over the factor ring is always faithful, as shown in \cite[Proposition 3.3]{BerghJorgensen1}. The following example shows that this is not the case here.

\begin{example}
Consider the diagram
$$\xymatrix@C=40pt{
B \ar[r]^{\cdot (x+y)} & B \ar[r]^{\cdot (x-qy)} & {_{\nu}B}}$$
As in the proof of Theorem \ref{thm:acyclic}, the composition of the two maps is given by 
$$1 \mapsto \nu (x+y) (x-qy) = (-q^{-1}x-qy)(x-qy) = w.$$
Moreover, the composition of the two maps
$$\xymatrix@C=40pt{
B \ar[r]^{\cdot (x-qy)} & {_{\nu}B} \ar[r]^{S_{\nu} \left (\cdot (x+y) \right )} & {_{\nu}B}}$$
is given by
$$1 \mapsto (x-qy) (x+y) = xy-qyx = w,$$
hence the top diagram is an object of $\Fact \left ( \F(B), S_{\nu}, \eta_w \right )$. Now the diagram
$$\xymatrix@C=40pt{
B \ar[r]^{\cdot (x+y)} \ar[d]^0 & B \ar[r]^{\cdot (x-qy)} \ar[d]^{\cdot w} & {_{\nu}B} \ar[d]^0 \\
B \ar[r]^{\cdot (x+y)} & B \ar[r]^{\cdot (x-qy)} & {_{\nu}B} }$$
commutes, and therefore represents a morphism in $\Fact \left ( \F(B), S_{\nu}, \eta_w \right )$. It is easily seen that this morphism is not nullhomotopic, and so it represents a nonzero morphism in the homotopy category $\HFact \left ( \F(B), S_{\nu}, \eta_w \right )$. However, this morphism trivially maps to zero by the functor $\HFact \left ( \F(B), S_{\nu}, \eta_w \right ) \to \Kac \left ( \F (A) \right )$.
\end{example}

Returning to the general theory, we now show that we can assign modules to twisted matrix factorizations functorially. Namely, since $A$ is a local selfinjective algebra, the homotopy category $\Kac( \F(A) )$ of acyclic complexes of finitely generated free left $A$-modules is equivalent to the stable module category $\stmod A$ of finitely generated left modules. The canonical equivalence
$$\Kac( \F(A) ) \longrightarrow \stmod A$$
maps an acyclic complex to the image of its degree zero map. The following is therefore an immediate corollary to Theorem \ref{thm:acyclic}.

\begin{corollary}\label{cor:stable}
Let $k$ be a field and $B$ the $k$-algebra $k \langle x,y \rangle / (x^2, y^2, xyx,yxy).$ Take an element $0 \neq q \in k$, put $w=xy-qyx$, and consider the quantum complete intersection 
$$A = B/(w) \simeq k \langle x,y \rangle / (x^2, xy-qyx, y^2).$$
Furthermore, let $\nu$ be the automorphism on $B$ defined by $\nu (x) = -q^{-1} x$ and $\nu (y) = -qy$. Finally, let $\F (B)$ be the category of finitely generated free left $B$-modules, $S_{\nu} \colon \F(B) \to \F(B)$ the twisting functor given by $\nu$, and $\eta_w \colon 1_{\F(B)} \to S_{\nu}$ the natural transformation with coordinates $\eta_F \colon F \to {_{\nu}F}$ given by $\eta_F(m) = wm$. Then reduction modulo $w$ induces a triangle functor 
$$T \colon \HFact \left ( \F(B), S_{\nu}, \eta_w \right ) \longrightarrow \stmod A,$$
where $\stmod A$ is the stable module category of finitely generated left $A$-modules. The image of a factorization $(F,G,f,g)$ in $\HFact \left ( \F(B), S_{\nu}, \eta_w \right )$ is the image of the map $F/wF \xrightarrow{\overline{f}} G/wG$.
\end{corollary}

As mentioned, we shall show that when $k$ is algebraically closed, then in a given dimension, almost all the finitely generated indecomposable left $A$-modules with bounded minimal projective resolutions are obtained from twisted matrix factorizations of the element $w$ over $B$. To do this, we use the classification of the indecomposable $A$ modules. Recall first that the \emph{complexity} of a finitely generated left $A$-module $M$ with minimal projective resolution
$$\cdots \to Q_2 \to Q_1 \to Q_0 \to M \to 0$$
is defined as
$$\cx_A M  \stackrel{\text{def}}{=} \inf \{ t \in \mathbb{N} \cup \{ 0 \} \mid \exists a \in \mathbb{R} \text{ such that } \dim_k Q_n \le an^{t-1} \text{ for } n \gg 0 \}.$$
The modules of complexity one are precisely the ones with bounded minimal projective resolutions. In \cite{BGMS}, a minimal projective bimodule resolution 
$$\mathbb{P}_A \colon \cdots \to P_2 \to P_1 \to P_0 \to A \to 0$$
of $A$ was constructed, with $\dim_k P_n = 16(n+1)$. Now if $M$ is a finitely generated left $A$-module, then the complex $\mathbb{P}_A \otimes_A M$ is acyclic, and is therefore a projective resolution of $M$. It may no longer be minimal, but its rate of growth is at most that of $\mathbb{P}_A$. Thus the complexity of any $A$-module is at most two.

Note that if the parameter $q$ is not a root of unity, then it was shown in \cite[Proposition 4.1]{Schulz} that there exist non-periodic $A$-modules of complexity one: the acyclic complex
$$\cdots \to A \xrightarrow{\cdot \left ( x+(-q)^{n+1}y \right )} A \xrightarrow{\cdot \left ( x+(-q)^{n}y \right )} A \xrightarrow{\cdot \left ( x+(-q)^{n-1}y \right )} A \to \cdots$$
gives rise to one class of such modules. However, if $q$ \emph{is} a root of unity, then by \cite[Proposition 4.1]{Schulz} again there are no such $A$-modules: in this case, the indecomposable modules of complexity one are precisely the periodic ones.

The classification of the indecomposable $A$-modules goes back to Kronecker. What follows is based on \cite[Section 4]{Schulz}, which again is based on \cite{Basev, Conlon1, Conlon2, HellerReiner}. It classifies the modules according to their complexities. 

\begin{fact}\label{fact:classification}
(1) Since complexity zero is the same as finite projective dimension, there is only one such indecomposable, namely $A$ itself. The modules of complexity two, that is, the ones having unbounded (but linearly growing) minimal projective resolutions, are the cokernels of the maps in the minimal complete resolution of the module $k$ over the algebra $k[x,y](x^2,y^2)$. These modules are all of odd dimension over $k$.

(2) The modules we are interested in are the ones of complexity one, that is, the ones with bounded minimal projective resolutions. These are all of even dimension over $k$, and \emph{in what follows we assume that $k$ is algebraically closed}. For each $n \ge 1$, the indecomposable $A$-modules of dimension $2n$ are
$$\{ C_n ( \lambda ) \mid \lambda \in k \cup \{ \infty \} \},$$
and these are described as follows. For $\lambda \in k$, denote by $J_n ( \lambda )$ the $n \times n$ Jordan matrix
$$\left ( \begin{array}{cccccc}
\lambda & 1 & 0 & \cdots & 0 & 0 \\
0 & \lambda & 1 & \cdots & 0 & 0 \\
\vdots & \vdots & \vdots & \ddots & \vdots & \vdots \\
\vdots & \vdots & \vdots & \ddots & \vdots & \vdots \\
0 & 0 & 0 & \cdots & \lambda & 1 \\
0 & 0 & 0 & \cdots & 0 & \lambda 
\end{array} \right )$$
with eigenvalue $\lambda$. The underlying vector space of $C_n ( \lambda )$ is $k^{2n}$, with the action of $x$ and $y$ given by the $2n \times 2n$ block matrices
\begin{eqnarray*}
x & \mapsto & \left ( 
\begin{array}{cc}
0 & J_n ( \lambda ) \\
0 & 0
\end{array}
\right ) \\
y & \mapsto & \left ( 
\begin{array}{cc}
0 & I_n \\
0 & 0
\end{array}
\right )
\end{eqnarray*}
where $I_n$ is the $n \times n$ identity matrix. The module $C_n ( \infty )$ has the same underlying vector space, with the action of $x$ and $y$ given by
\begin{eqnarray*}
x & \mapsto & \left ( 
\begin{array}{cc}
0 & I_n \\
0 & 0
\end{array}
\right ) \\
y & \mapsto & \left ( 
\begin{array}{cc}
0 & J_n(0) \\
0 & 0
\end{array}
\right )
\end{eqnarray*}
Note that every one of these modules has $n$ independent generators.
\end{fact}

We are now ready to state and prove the main result.

\begin{theorem}\label{thm:main}
Let $k$ be an algebraically closed field, $q \in k$ a nonzero element, and $A$ the quantum complete intersection
$$A = \langle x,y \rangle / (x^2, xy-qyx, y^2).$$ Furthermore, let $B, w, \nu$ and the triangle functor
$$T \colon \HFact \left ( \F(B), S_{\nu}, \eta_w \right ) \longrightarrow \stmod A$$
be as in \emph{Corollary \ref{cor:stable}}, and 
$$\{ C_n ( \lambda ) \mid n \in \mathbb{N}, \lambda \in k \cup \{ \infty \} \}$$
the indecomposable left $A$-modules with bounded projective resolutions. Then $C_n ( \lambda )$ is in the image of $T$ if and only if $\lambda \notin \{ 0, \infty \}$. Thus, for $0 \neq \lambda \in k$ there is a factorization $(F,G,f,g) \in \HFact \left ( \F(B), S_{\nu}, \eta_w \right )$ with  $C_n ( \lambda )$ isomorphic to the image of the map $F/wF \xrightarrow{\overline{f}} G/wG$, whereas for $C_n (0)$ and $C_n ( \infty )$ there are no such factorizations.
\end{theorem}

\begin{proof}
We start by proving the following. Let
$$\xymatrix@C=30pt{
B^n \ar[r]^{f} & B^n \ar[r]^{g} & {_{\nu}B^n}}$$ 
be a rank $n$ factorization in $\HFact \left ( \F(B), S_{\nu}, \eta_w \right )$, and suppose that the matrix for $f$ is of the form $xI + yC + xyC' + yxC''$, with $C$ invertible (the matrices $I, C, C', C''$ are $n \times n$ matrices over $k$, with $I$ the identity matrix). We claim that the image of the map $F/wF \xrightarrow{\overline{f}} G/wG$ is isomorphic to the left $A$-module with underlying vector space $k^{2n}$, and with the action of $x$ and $y$ given by
\begin{eqnarray*}
x & \mapsto & \left ( 
\begin{array}{cc}
0 & qC \\
0 & 0
\end{array}
\right ) \\
y & \mapsto & \left ( 
\begin{array}{cc}
0 & I \\
0 & 0
\end{array}
\right )
\end{eqnarray*}

To see this, note that the image of the map $f$ is generated as a $B$-module by the $n$ rows $\zeta_1, \dots, \zeta_n$ of the matrix $xI + yC + xyC' + yxC''$. We describe the action of $x$ and $y$ on these generators, in terms of the standard generators $b_1, \dots, b_n$ of $B^n$, where $b_i$ is the row vector with a single nonzero entry, the unit of $B$. First, since $y^2 = 0 = yxy$ in $B$, we see that $y \left ( xI + yC + xyC' + yxC'' \right ) = yxI$, and so
$$y \zeta_i = yx b_i.$$
Similarly, since $x^2 = 0 = xyx$, we obtain $x \left ( xI + yC + xyC' + yxC'' \right ) = xyC$, hence
$$x \zeta_i = xy \sum_{j=1}^n c_{ij}b_j,$$
where $C = (c_{ij})$. Rewriting $xy$ as $qyx + w$, we obtain
\begin{eqnarray*}
x \zeta_i & = & (qyx + w) \sum_{j=1}^n c_{ij}b_j \\
& = & \sum_{j=1}^n qc_{ij} (yxb_j) + w \sum_{j=1}^n c_{ij}b_j  \\
& = & \sum_{j=1}^n qc_{ij} (y \zeta_j) + w \sum_{j=1}^n c_{ij}b_j,
\end{eqnarray*}
where we have used the equality $y \zeta_j = yx b_j$ from above. All this shows that when we factor out $w$, then the image of the map $\overline{f}$ is spanned as a vector space by the cosets of 
$$\zeta_1, \dots, \zeta_n, y \zeta_1, \dots, y \zeta_n.$$
The equality $y \zeta_i = yx b_i$ gives $x (y \zeta_i) = 0 = y (y \zeta_i$, and using it once more we see that the $2n$ cosets are linearly independent. The action of $x$ and $y$ on the vector space they generate are given by the two matrices, hence the claim follows.

Having proved the claim, let $\lambda$ be a nonzero element in $k$. Then the $n \times n$ Jordan matrix $J_n ( \lambda )$ is invertible, and so by Remark \ref{rem:matrices} there exists a factorization
$$\xymatrix@C=30pt{
B^n \ar[r]^{f} & B^n \ar[r]^{g} & {_{\nu}B^n}}$$ 
in $\HFact \left ( \F(B), S_{\nu}, \eta_w \right )$ with $xI + yq^{-1} J_n ( \lambda )$ as the matrix for the map $f$. By the claim above, the image of the map $\overline{f}$ has underllying vector space $k^{2n}$, and with the action of $x$ and $y$ given by
\begin{eqnarray*}
x & \mapsto & \left ( 
\begin{array}{cc}
0 & J_n ( \lambda ) \\
0 & 0
\end{array}
\right ) \\
y & \mapsto & \left ( 
\begin{array}{cc}
0 & I \\
0 & 0
\end{array}
\right )
\end{eqnarray*}
This is precisely the $A$-module $C_n ( \lambda )$.

Finally, we show that the $A$-modules $C_n (0)$ and $C_n ( \infty )$ cannot be obtained from twisted matrix factorizations over $B$, for any $n$. Namely, take any nonzero rank $n$ factorization 
$$\xymatrix@C=30pt{
F \ar[r]^{f} & G \ar[r]^{g} & {_{\nu}F}}$$ 
in $\HFact \left ( \F(B), S_{\nu}, \eta_w \right )$. By Remark \ref{rem:matrices}, the matrix for $f$ is of the form $xC_1+yC_2+xyC_3+yxC_4$, where the $C_i$ are $n \times n$ matrices over $k$, and with $C_1$ and $C_2$ invertible. Consider the factorization $(F,G, C_1^{-1}f, gC_1)$, in which the matrix for the map $C_1^{-1}f$ is that for $f$ multiplied on the left with $C_1^{-1}$, whereas the matrix for $gC_1$ is that for $g$ multiplied on the right with $C_1$. The diagram
$$\xymatrix@C=30pt{
F \ar[r]^{f} \ar[d]^{\cdot C_1} & G \ar[r]^{g} \ar[d]^{\cdot I} & {_{\nu}F} \ar[d]^{\cdot C_1} \\
F \ar[r]^{C_1^{-1}f} & G \ar[r]^{gC_1} & {_{\nu}F}}$$
shows that the two factorizations are isomorphic in $\HFact \left ( \F(B), S_{\nu}, \eta_w \right )$, and they are therefore mapped by the functor $T$ to isomorphic $A$-modules. The matrix for the map $C_1^{-1}f$ is $x+ yC_1^{-1}C_2+xyC_1^{-1}C_3+yxC_1^{-1}C_4$, hence from the claim we proved the image of these factorizations under $T$ is isomorphic to the following $A$-module: the underlying vector space is $k^{2n}$, and the action of $x$ and $y$ are given by
\begin{eqnarray*}
x & \mapsto & \left ( 
\begin{array}{cc}
0 & qC_1^{-1}C_2 ( \lambda ) \\
0 & 0
\end{array}
\right ) \\
y & \mapsto & \left ( 
\begin{array}{cc}
0 & I \\
0 & 0
\end{array}
\right )
\end{eqnarray*}
Both these matrices have rank $n$. However, from Fact \ref{fact:classification} we see that for $C_n(0)$, the matrix that defines the action of $x$ has rank $n-1$, as does the matrix that defines the action of $y$ on $C_n( \infty )$. Namely, these actions are given by $2n \times 2n$ block matrices with $J_n (0)$ as the nonzero $n \times n$ block. This shows that the $A$-modules $C_n (0)$ and $C_n ( \infty )$ cannot be obtained from twisted matrix factorizations over $B$.
\end{proof}

\end{document}